\documentclass[english]{amsart}
\usepackage{amsmath}

\usepackage{amssymb}
\usepackage{svn}

\usepackage{amscd}
\usepackage[all,cmtip,line]{xy} 

\newcommand{\fN}{\mathfrak{N}}

\newcommand{\oR}{\overline{R}}

\newcommand{\Xbar}{\overline{X}}
\newcommand{\Ybar}{\overline{Y}}

\newcommand{\arr}{{\; \rightarrow \;}}

\SVN $LastChangedDate: 2012-06-27 18:49:49 +0100 (Wed, 27 Jun 2012) $
\SVN $LastChangedRevision: 21 $

\def\CC{\mathbb{C}}

\def\FF{\mathbb{F}}
\def\GG{\mathbb{G}}

\def\QQ{\mathbb{Q}}

\def\ZZ{\mathbb{Z}}

\newcommand{\gera}{{\frak{a}}}

\newcommand{\gerd}{{\frak{d}}}

\newcommand{\gerl}{{\frak{l}}}
\newcommand{\germ}{{\frak{m}}}

\newcommand{\uA}{{\underline{A}}}

\newcommand{\calO}{{\mathcal{O}}}

\newlength{\ownl}

\newcommand{\Frob}{{\operatorname{Frob}}}
\newcommand{\Fr}{{\operatorname{Fr}}}
\newcommand{\Gal}{{\operatorname{Gal}\,}}

\newcommand{\pr}{{\operatorname{pr}\,}}

\newcommand{\tr}{{\operatorname{tr}\,}}

\newcommand{\GL}{\operatorname{GL}}

\newcommand{\PGL}{\operatorname{PGL}}
\newcommand{\PSL}{\operatorname{PSL}}

\newcommand{\C}{{\mathbb{C}}}

\newcommand{\F}{{\mathbb{F}}}

\newcommand{\Q}{{\mathbb{Q}}}

\newcommand{\Z}{{\mathbb{Z}}}

\newcommand{\cF}{\mathcal{F}}
\newcommand{\cG}{\mathcal{G}}

\newcommand{\cO}{\mathcal{O}}

\newcommand{\cS}{\mathcal{S}}

\newcommand{\tS}{\widetilde{{S}}}

\newcommand{\tv}{{\widetilde{{v}}}}

\newcommand{\varepsilonbar  }{\overline{\varepsilon}}

 \newcommand{\rhobar   }{{\overline{\rho}}}

\newcommand{\gammat   }{\widetilde{\gamma}}

\newcommand{\proj}{{\operatorname{proj}}}

 \newcommand{\p}{\mathfrak{p}}

\newcommand{\Qp}{\Q_p}

\newcommand{\Qpbar}{\overline{\Q}_p}

\newcommand{\Fpbar}{\overline{\F}_p}

\usepackage{hyperref} 

\usepackage{amsthm}

\newtheorem{ithm}{Theorem}
\newtheorem{thm}{Theorem}[subsection]

\newtheorem{cor}[thm]{Corollary}
 \newtheorem{lemma}[thm]{Lemma}
 
\newtheorem{prop}[thm]{Proposition}
\newtheorem{conj}[thm]{Conjecture} \theoremstyle{definition}
 \theoremstyle{definition}
 \theoremstyle{remark}

\numberwithin{equation}{subsection}

\theoremstyle{definition}

\begin{document}
\title[Companion Forms in Parallel Weight One]{Companion Forms in Parallel Weight One}

\author{Toby Gee} \email{toby.gee@imperial.ac.uk} \address{Department of
  Mathematics, Imperial College London} \author{Payman
  Kassaei}\email{payman.kassaei@kcl.ac.uk}\address{Department of
  Mathematics, King's College London}  
\begin{abstract}Let $p>2$ be prime, and let $F$ be a totally real
  field in which $p$ is unramified. We give a sufficient criterion for
  a mod $p$ Galois representation to arise from a mod $p$ Hilbert
  modular form of parallel weight one, by proving a ``companion
  forms'' theorem in this case. The techniques used are a mixture of
  modularity lifting theorems and geometric methods. As an
  application, we show that Serre's conjecture for $F$ implies Artin's
  conjecture for totally odd two-dimensional representations over $F$.
\end{abstract}
\maketitle\section{Introduction}The weight part of Serre's conjecture has been
  much-studied over the last two decades, and while the original
  problem has been resolved, a great deal remains to be proved in more general
  settings. In the present paper we address the question of the weight
  part of Serre's conjecture for totally real fields. Here one has the
  Buzzard--Diamond--Jarvis conjecture \cite{bib:BDJ} and its various
  generalisations, much of which has now been established
  (cf. \cite{gee08serrewts}, \cite{blgg11-serre-weights-for-U2}). One
  case that has not (so far as we are aware) been considered at all
  over totally real fields is the case of forms of (partial) weight
  one. This case is markedly different to the case of higher weights,
  for the simple reason that mod $p$ modular forms of weight one
  cannot necessarily be lifted to characteristic zero modular forms of
  weight one; as the methods of \cite{gee08serrewts} and
  \cite{blgg11-serre-weights-for-U2} are centered around modularity
  lifting theorems, and in particular depend on the use of lifts to
  characteristic zero modular forms, they cannot immediately say
  anything about the weight one case.

In this paper we generalise a result of Gross \cite{MR1074305}, and
prove a companion forms theorem for Hilbert modular forms of parallel
weight one in the unramified $p$-distinguished case. To explain what this means,
and how it (mostly) resolves the weight one part of Serre's conjecture
for totally real fields, we return to the case of classical modular
forms. Serre's original formulation \cite{MR885783} of his
conjecture only considered mod $p$ modular forms which lift to
characteristic zero, and in particular ignored the weight one
case. However, Serre later observed that one could further refine his
conjecture by using Katz's definition \cite{MR0447119} of mod $p$
modular forms, and thus include weight one forms. He then conjectured
that a modular Galois representation should arise from a weight one
form (of level prime to $p$) if and only if the Galois representation
is unramified at $p$. The harder direction is to prove that the Galois
representation being unramified implies that there is a weight one
form; this was proved by Gross \cite{MR1074305}, under the further
hypothesis that the eigenvalues of a Frobenius element at $p$ are
distinct (i.e. we are in the $p$-distinguished case). It is this
result that we generalise in this paper, proving the following
theorem.

\begin{ithm}\label{thm: main result, introduction version}
  Let $p>2$ be prime, let $F$ be a totally real field in which $p$ is
  unramified, and let $\rhobar:G_F\to\GL_2(\Fpbar)$ be an irreducible
  modular representation such that $\rhobar|_{G_{F(\zeta_p)}}$ is
  irreducible.   If $p=3$ (respectively $p=5$), assume further that the projective image of
  $\rhobar(G_{F(\zeta_p)})$ is not conjugate to $\PSL_2(\F_3)$ (respectively $\PSL_2(\F_5)$ or
  $\PGL_2(\F_5)$).

Suppose that
  for each place $v|p$, $\rhobar|_{G_{F_v}}$ is unramified, and that
  the eigenvalues of $\rhobar(\Frob_v)$ are distinct.

 Then there is a mod $p$ Hilbert modular form $f$ of
  parallel weight $1$ and level prime to $p$ such that $\rhobar_f\cong\rhobar$.
  \end{ithm}
  (See Theorem \ref{thm: main result}, and see the body of the paper
  for any unfamiliar notation or terminology.) The condition on
  $\rhobar|_{G_{F(\zeta_p)}}$ is mild, and the only other hypothesis
  which is not expected to be necessary (other than that $p$ is
  unramified in $F$, which we assume for our geometric arguments) is
  that the eigenvalues of $\rhobar(\Frob_v)$ are distinct for all
  $v|p$. This condition appears to be essential to our method, as we
  explain below.

Our method of proof is a combination of modularity lifting theorem
  techniques and geometric methods.

  The first part of the argument, using modularity lifting theorems
  to produce Hilbert modular forms of parallel weight $p$, was carried
  out in \cite{gee051} under some additional hypotheses (in
  particular, it was assumed that $\rhobar$ arose from an ordinary
  Hilbert modular form), and in Section \ref{sec: BLGG} below we use
  the techniques of \cite{BLGGT} (which involve the use of automorphy
  lifting theorems for rank 4 unitary groups) to remove these
  hypotheses. This gives us $2^n$ Hilbert modular forms of parallel
  weight $p$ and level prime to $p$, where there are $n$ places of $F$
  above $p$, corresponding to the different possible choices of
  Frobenius eigenvalues at places above $p$.   In Section \ref{section:
    weight one} we take a suitable linear combination of these forms,
  and show that it is divisible by the Hasse invariant of  parallel weight $p-1$, by explicitly calculating  the $p$-th power of the quotient.  It is easy to show that the quotient is the sought-after Hilbert modular form of parallel  weight one.  If we do not assume that $\rhobar$ has
  distinct Frobenius eigenvalues at each place dividing $p$, the
  weight one form we obtain in this manner is actually zero.

In Section \ref{sec: Artin} we give an application of our
main theorem to Artin's conjecture, generalising the results (and
arguments) of \cite{MR1434905} to prove the following result, which
shows that the weak form of Serre's conjecture for totally real fields
implies the strong form of Artin's conjecture for totally odd
two-dimensional representations.

\begin{ithm}
  Let $F$ be a totally real field. Assume that every irreducible,
  continuous and totally odd representation
  $\rhobar:G_F\to\GL_2(\Fpbar)$ is modular, for every prime $p$. Then
  every irreducible, continuous and totally odd representation
  $\rho:G_F\to\GL_2(\C)$ is modular.
\end{ithm}
(See Theorem \ref{thm: Serre implies Artin}.)  Finally, we remark that
it is possible that our results can be applied to Artin's conjecture
more directly (that is, without assuming Serre's conjecture), as an
input to modularity lifting theorems in parallel weight one; see the
forthcoming work of Calegari--Geraghty for some results in this
direction. Additionally, it would be of interest to generalise our
results to forms of partial weight one; the geometric techniques we
use in this paper amount to determining the intersection of the
kernels of the $\Theta$-operators of \cite{AG}, and it is possible
that a determination of the kernels of the individual
$\Theta$-operators could shed some light on this.

We are grateful to Shu Sasaki for suggesting that our results could be
used to generalise the arguments of \cite{MR1434905} to totally real
fields. We are also grateful to Kevin Buzzard for several helpful
conversations, and to Frank Calegari for asking a question which led
to our writing this paper.

\subsection{Notation} If $M$ is a field, we let $\overline{M}$ denote
an algebraic closure of $M$, and we let $G_M:=\Gal(\overline{M}/M)$
denote its absolute Galois group. Let $p$ be a prime number, and let
$\varepsilon$ denote the $p$-adic cyclotomic character; our choice of
convention for Hodge--Tate weights is that $\varepsilon$ has all
Hodge--Tate weights equal to $1$. Let $F$ be a totally real field and
$f$ a cuspidal Hilbert modular eigenform of parallel weight $k$. If
$v$ is a finite place of $F$ which is coprime to the level of $f$,
then we have, in particular, the usual Hecke operator $T_v$
corresponding to the double coset \[\GL_2(\cO_{F_v})
    \begin{pmatrix}
      \varpi_v&0\\0&1
    \end{pmatrix}\GL_2(\cO_{F_v}),\]where $\varpi_v$ is a uniformiser of $\cO_{F_v}$,
    the ring of integers of $F_v$. 

    There is a Galois representation $\rho_f:G_F\to\GL_2(\Qpbar)$
    associated to $f$; we adopt the convention that if $v\nmid p$ is
    as above, and $\Frob_v$ is an \emph{arithmetic} Frobenius element
    of $G_{F_v}$ then $\tr\rho_f(\Frob_v)$ is the $T_v$-eigenvalue of
    $f$, so that in particular the determinant of $\rho_f$ is a finite
    order character times $\varepsilon^{k-1}$.

We say that
    $\rhobar:G_F\to\GL_2(\Fpbar)$ is \emph{modular} if it arises as
    the reduction mod $p$ of the Galois representation
    $\rho_f:G_F\to\GL_2(\Qpbar)$ for some $f$.

\section{Modularity lifting in weight $p$}\label{sec: BLGG}\subsection{}In this section we apply the
modularity lifting techniques first used in \cite{gee051} to produce
companion forms in (parallel) weight $p$. We make use of the
techniques of \cite{BLGGT} in order to weaken the  hypotheses
(for example, to avoid the necessity of an assumption of
ordinarity). In this section, we do not assume that $p$ is unramified
in $F$. Note that the definition of a mod $p$ modular form is recalled
in Section~\ref{section: weight one} below.

\begin{thm}\label{thm: lifting to weight p by blgg}
  Let $p>2$ be prime, let $F$ be a totally real field, and let
  $\rhobar:G_F\to\GL_2(\Fpbar)$ be an irreducible modular
  representation such that $\rhobar|_{G_{F(\zeta_p)}}$ is irreducible.
  If $p=3$ (respectively $p=5$), assume further that the projective image of
  $\rhobar(G_{F(\zeta_p)})$ is not conjugate to $\PSL_2(\F_3)$ (respectively $\PSL_2(\F_5)$ or
  $\PGL_2(\F_5)$).

Suppose that
  for each place $v|p$, $\rhobar|_{G_{F_v}}$ is unramified, and in
  fact suppose that \[\rhobar|_{G_{F_v}}\cong
  \begin{pmatrix}
    \lambda_{\alpha_{1,v}} &0\\0&\lambda_{\alpha_{2,v}}
  \end{pmatrix}\]where $\lambda_x$ is the unramified character sending
  an arithmetic Frobenius element to $x$. For each place $v|p$, let
  $\gamma_v$ be a choice of one of $\alpha_{1,v}$,
  $\alpha_{2,v}$. Let $N$ be an integer coprime to $p$ and divisible
  by the Artin conductor of $\rhobar$. Then there is a mod $p$ Hilbert modular eigenform $f$ of
  parallel weight $p$ such that:
  \begin{itemize}
  \item $f$ has level $\Gamma_1(N)$; in particular, $f$ has level prime
    to $p$.
  \item $\rhobar_f\cong\rhobar$.
  \item $T_v f =\gamma_vf$ for each place $v|p$.
  \end{itemize}

\end{thm}
\begin{proof} It suffices to prove that there is a (characteristic zero) Hilbert
  modular eigenform $\cF$ of parallel weight $p$ such that $\cF$ has level
$N$, the Galois representation $\rho_{\cF}$ associated to
  $\cF$ satisfies $\rhobar_{\cF}\cong\rhobar$, and for each place
  $v|p$ we have $T_v\cF=\gammat_v \cF$ for some lift $\gammat_v$ of
  $\gamma_v$. (We remark that in the argument below, we will feel free
  to let $\gammat_v$ denote \emph{any} lift of $\gamma_v$, rather than
  a fixed lift.) By local-global compatibility at places dividing $p$
  (cf. Theorem 4.3 of \cite{kisinpst}), it is enough to find a lift
  $\rho$ of $\rhobar$ which is automorphic, which is minimally
  ramified outside $p$, and has the further property that for each place
  $v|p$ we have \[\rho|_{G_{F_v}}\cong
  \begin{pmatrix}
    \varepsilon^{p-1}\lambda_{x_v} & *\\ 0 & \lambda_{\gammat_v}
  \end{pmatrix}\]for some $x_v$ and some $\gammat_v$ (such a
  representation is automatically crystalline with Hodge--Tate weights
  $0$, $p-1$ at each place $v|p$ and thus corresponds to a Hilbert
  modular form of parallel weight $p$ and level prime to $p$).

  The existence of such a representation $r$ is a straightforward
  application of the results of \cite{blgg11-serre-weights-for-U2}, as
  we now explain. This argument is very similar to the proof of
  Proposition 6.1.3 of \cite{blggord}. Firstly, choose a quadratic
  imaginary CM extension $F_1/F$ which splits at all places dividing
  $p$ and all places where $\rhobar$ is ramified, and which is
  linearly disjoint from $\overline{F}^{\ker\rhobar}(\zeta_p)$ over
  $F$. Let $S$ denote a finite set of places of $F$, consisting of the
  infinite places, and the union of set of places of $F$ at which
  $\rhobar$ is ramified and the places which divide $p$. From now on
  we will consider $\rhobar$ as a representation of $G_{F,S}$, the
  Galois group of the maximal extension of $F$ unramified outside of
  $S$. Let $\chi$ be the Teichm\"uller lift of
  $\varepsilonbar^{1-p}\det\rhobar$.

  Fix a finite extension $E/\Qp$ with ring of integers $\cO$ and
  residue field $\F$ such that $\rhobar$ is valued in $\GL_2(\F)$. For
  each finite place $v$ of $F$, let $R^\chi_{G_{F_v}}$ denote the
  universal $\cO$-lifting ring for lifts of $\rhobar|_{G_{F_v}}$ of
  determinant $\chi\varepsilon^{p-1}$. By (for example) Lemma 3.1.8 of \cite{GG}, for
  each place $v|p$ of $F$ there is a quotient
  $R_{G_{F_v}}^{\chi,\gamma}$ of $R_{G_{F_v}}^\chi$ whose
  $\Qpbar$-points correspond precisely to those lifts of
  $\rhobar|_{G_{F_v}}$ which are conjugate to a representation of the
  form \[\begin{pmatrix} \varepsilon^{p-1}\chi/\lambda_{\gammat_v} & *\\
    0 & \lambda_{\gammat_v}
  \end{pmatrix}\] for some $\gammat_v$ lifting $\gamma_v$. For each
  finite place $v\in S$ with $v\nmid p$, let $\oR_{G_{F_v}}^\chi$ be a
  quotient of $R_{G_{F_v}}^\chi$ corresponding to an irreducible
  component of $R_{G_{F_v}}[1/p]$, the points of which correspond to
  lifts of $\rhobar|_{G_{F_v}}$ with the same Artin conductor as
  $\rhobar|_{G_{F_v}}$. Let $R^{\chi,\gamma}$ denote the universal
  deformation ring for deformations of $\rhobar$ of determinant
  $\chi\varepsilon^{p-1}$, which have the additional property that for
  each place $v|p$, the deformation corresponds to a point of
  $R_{G_{F_v}}^{\chi,\gamma}$, and for each finite place $v\in S$,
  $v\nmid p$, it corresponds to a point of $\oR_{G_{F_v}}^{\chi}$. In
  order to construct the representation $r$ that we seek, it is enough
  to find a $\Qpbar$-point of $R^{\chi,\gamma}$ that is
  automorphic. We will do this by showing that $R^{\chi,\gamma}$ is a
  finite $\cO$-algebra of dimension at least one (so that it has
  $\Qpbar$-points), and that all its $\Qpbar$-points are automorphic.

  We can and do extend $\rhobar|_{G_{F_1}}$ to a representation
  $\rhobar:G_F\to\cG_2(\F)$, where $\cG_2$ is the group scheme
  introduced in Section 2.1 of \cite{cht} (cf. Section 3.1.1 of
  \cite{blggord}). In the notation of Section 2.3 of \cite{cht}, we
  let $\tS$ be a set of places of $F_1$ consisting of exactly one
  place $\tv$ above each place $v$ of $S$, and we let $\cS$ denote the
  deformation
  problem \[(F_1/F,S,\tS,\cO,\rhobar,\varepsilon^{p-1}\chi,\{\oR^\chi_{G_{F_{1,\tv}}}\}_{v\in
    S, v\nmid p},\{R_{G_{F_{1,\tv}}}^{\chi,\gamma}\}_{v|p}).\] Let
  $R_\cS$ denote the corresponding universal deformation ring. Exactly
  as in section 7.4 of \cite{GG}, the process of ``restriction to
  $G_{F_1}$ and extension to $\cG_2$'' makes $R^{\chi,\gamma}$ a
  finite $R_\cS$-module in a natural way. By Proposition 3.1.4 of
  \cite{gee061} we have $\dim R^{\chi,\gamma}\ge 1$, so that (by
  cyclic base change for $\GL_2$) it suffices to show that $R_\cS$ is
  finite over $\cO$, and that all its $\Qpbar$-points are automorphic.

  By Proposition A.2.1 of \cite{blgg11-serre-weights-for-U2} and our
  assumptions on $\rhobar|_{G_{F(\zeta_p)}}$,
  $\rhobar(G_{F(\zeta_p)})$ is adequate in the sense of
  \cite{jack}. By Theorems 7.1 and 10.1 of \cite{jack}, it is enough
  to check that $R_\cS$ has an automorphic $\Qpbar$-point. We claim
  that we can do this by applying Theorem A.4.1 of
  \cite{blgg11-serre-weights-for-U2} to $\rhobar$. The only hypotheses
  of {\em loc. cit.} that are not obviously satisfied are those
  pertaining to potential diagonalizability. By Lemma 3.1.1 of
  \cite{blgg11-serre-weights-for-U2}, we may choose a finite solvable
  extension $F'/F$ of CM fields which is linearly disjoint from
  $\overline{F}^{\ker\rhobar}(\zeta_p)$, such that $\rhobar|_{G_{F'}}$
  has an automorphic lift which is potentially diagonalizable at all
  places dividing $p$. All $\Qpbar$-points of each
  $R_{G_{F_v}}^{\chi,\gamma}$ are potentially diagonalizable by Lemma
  1.4.3 of \cite{BLGGT}, so Theorem A.4.1 of
  \cite{blgg11-serre-weights-for-U2} produces a $\Qpbar$-point of
  $R_{\cS}$ which is automorphic upon restriction to $G_{F'}$. Since
  $F'/F$ is solvable, the result follows by solvable base change
  (Lemma 1.4 of \cite{BLGHT}).
\end{proof}

\section{Weight one forms}\label{section: weight one}\subsection{}

Let $p>2$ be a prime number.  Let $F/\mathbb{Q}$ be a totally real
field of degree $d>1$,  $\mathcal{O}_F$ its ring of integers, and
$\gerd_F$ its different ideal.    Let $\mathbb{S}=\{v|p\}$ be the set of all primes lying above $p$. We assume that $p$ is unramified in $F$. Let $N>3$ be an integer prime to $p$.

In this section we make our geometric arguments. We begin by
recalling some standard definitions. Let $K$ be a finite extension of
$\QQ_p$ (which we will assume to be sufficiently large without further
comment), and let $\calO_K$, $\kappa$ denote, respectively, its ring
of integers and its residue field. Let $X/\calO_K$ be the Hilbert
modular scheme representing the functor which associates to an
$\calO_K$-scheme $S$, the set of all  polarized abelian schemes with
real multiplication and $\Gamma_{00}(N)$-structure
$\underline{A}/S=(A/S,\iota,\lambda,\alpha)$ as follows: 
 \begin{itemize}
 \item $A$ is an abelian scheme of relative
dimension $d$ over  $S$;
\item the real multiplication $\iota\colon\mathcal{O}_F \hookrightarrow {\rm End}_S(A)$ is a ring
homomorphism endowing $A$ with an action of $\mathcal{O}_F$;
\item  the map $\lambda$ is a polarization as in \cite{DP};
\item  $\alpha$ is a rigid
$\Gamma_{1}(N)$-level structure,  that is, $\alpha\colon \mu_N
\otimes_{\ZZ} \gerd_F^{-1} \arr A$, an $\calO_F$-equivariant
closed immersion of group schemes. 
 \end{itemize}

 Let $X_K/K$, $\Xbar/\kappa$ respectively denote the generic and the special fibres of $X$.
 Let $\widetilde{X}$ denote a toroidal compactification of $X$. Similarly, we define $\widetilde{X}_K$ and $\widetilde{\Xbar}$.

 Let $Y$ be the scheme representing the functor which associates to an $\calO_K$-scheme $S$, the set of all $(\underline{A}/S,C)=(A/S,\iota,\lambda,\alpha,C)$, where $(A/S,\iota,\lambda,\alpha)$ is as above, and $C$ is an $\calO_F$-invariant
isotropic finite flat subgroup scheme of $A[p]$ of order $p^g$.  Let $\widetilde{Y}$ denote a toroidal compactification of $Y$ obtained using the same choices of polyhedral decompositions as for $\Xbar$.  We introduce the notation $Y_K,\Ybar,\widetilde{Y}_K,\widetilde{\Ybar}$ in the same way as we did for $X$. The ordinary locus in $\Xbar$ is denoted by $\Xbar^{ord}$. It is Zariski dense in $\Xbar$.

There are two finite flat maps
\[
\pi_{1},\pi_2:\widetilde{Y} \arr \widetilde{X},
\]
where $\pi_1$ forgets the subgroup $C$, and $\pi_2$ quotients out by $C$.  We define the Atkin--Lehner involution $w:\widetilde{Y} \rightarrow \widetilde{Y}$ to be the map which sends $(\uA,C)$ to $(\uA/C, A[p]/C)$; it is an automorphism of $\widetilde{Y}$.  We have $\pi_2=\pi_1 \circ w$. We also define the Frobenius section $s:\widetilde{\Xbar} \rightarrow \widetilde{\Ybar}$ which sends $\uA$ to $(\uA,{\rm Ker}(\Frob_{A}))$. Our convention is to use the same notation to denote maps between the various versions of $X,Y$.

Let $\epsilon: \underline{\mathcal A}^{\rm univ} \arr X$ be the universal abelian scheme. Let 
 \[
 \underline{\Omega}=\epsilon_\ast \Omega^1_{\underline{\mathcal A}^{\rm
univ}/X}
\]
be the Hodge bundle on $X$.  Since $p$ is assumed unramified in $F$, $\underline{\Omega}$ is a locally free $(\calO_F \otimes_\ZZ \calO_X)$-module of rank one. We define $\underline{\omega}=\wedge^d \underline{\Omega}$.  The sheaf $\underline{\omega}$ naturally extends to $\widetilde{X}$ as an invertible sheaf. Let $\epsilon^\prime: \underline{\mathcal B}^{\rm univ} \arr Y$ be the universal abelian scheme over $Y$ with the designated subgroup $\mathcal C$. We have
\[
\pi_1^\ast\underline{\omega}=\wedge^d \epsilon_\ast \Omega^1_{{\mathcal B}^{\rm
univ}/Y},
\]
\[
\pi_2^\ast\underline{\omega}=\wedge^d \epsilon_\ast \Omega^1_{     ({\mathcal B}^{\rm
univ}/{\mathcal C})/Y   }.
\]
Let 
\[
\rm{pr}^\ast: \pi_1^\ast \underline{\omega} \rightarrow \pi_2^\ast\underline{\omega}
\]
 denote the pullback under the natural projection  ${\rm{pr}}: \underline{\mathcal B}^{\rm univ} \rightarrow \underline{\mathcal B}^{\rm univ}/{\mathcal C}$. We will often denote $\pi_1^\ast \underline{\omega}$ by $\underline{\omega}$.

Let $R$ be an $\calO_K$-algebra. A (geometric) Hilbert modular form of parallel weight $k \in \ZZ$ and level $\Gamma_1(N)$ defined over $R$ is a section of $\underline{\omega}^k$ over $X\otimes_{\calO_K} R$. Every such section extends to $\widetilde{X} \otimes_{\calO_K} R$ by the Koecher principle. We denote the space of such forms by $M_k(\Gamma_1(N),R)$, and the subspace of cuspforms (those sections vanishing on the cuspidal locus) by $S_k(\Gamma_1(N),R)$. If $R$ is a $\kappa$-algebra, then elements of $M_k(\Gamma_1(N),R)$ are referred to as \emph{mod $p$ Hilbert modular forms}. Every such form is a section of $\underline{\omega}^k$ over $\Xbar\otimes_{\kappa} R$, and extends automatically to $\widetilde{\Xbar} \otimes_{\kappa} R$. 

Given any fractional ideal $\gera$ of $\calO_F$, let $X_\gera$  denote the subscheme of $X$  where the polarization module of the abelian scheme is isomorphic to $(\gera,\gera^+)$ as a module with notion of positivity. Then $X_\gera$ is a connected component of  $X$, and every connected component of $X$ is of the form $X_\gera$ for some $\gera$. The same statement is true for $Y_\gera$, $Y$.   Let 
 \[
Ta_{\gera}=(\underline{\GG_m\! \otimes\! \gerd_F^{-1})/q^{\gera^{-1}}}
\]
denote a cusp on $X_\gera$, where underline indicates the inclusion of standard PEL structure.  Pick $c \in p^{-1}\gera^{-1}-\gera^{-1}$. Then
\[
Ta_{\gera,c}=(Ta_\gera,<\! q^c\!>)
\]
is a cusp on $Y_\gera$, where $<\!q^c\!>$ denotes the $\calO_F$-submodule of $Ta_\gera[p]$ generated by $q^c$.

We now prove our main result, the following theorem.

\begin{thm}\label{thm: main result} Let $p>2$ be a prime which is
  unramified in $F$, a totally real field.  Let
  $\rhobar:G_F\to\GL_2(\Fpbar)$ be an irreducible modular
  representation such that $\rhobar|_{G_{F(\zeta_p)}}$ is irreducible.
  If $p=3$ (respectively $p=5$), assume further that the projective
  image of $\rhobar(G_{F(\zeta_p)})$ is not conjugate to
  $\PSL_2(\F_3)$ (respectively $\PSL_2(\F_5)$ or $\PGL_2(\F_5)$).

Suppose that
  for each prime $v|p$, $\rhobar|_{G_{F_v}}$ is unramified, and that
  the eigenvalues of $\rhobar(\Frob_v)$ are distinct.

 Then there is a mod $p$ Hilbert modular form $h$ of
  parallel weight $1$ and level prime to $p$ such that
  $\rhobar_h\cong\rhobar$. Furthermore, $h$ can be chosen to have
  level bounded in terms of the Artin conductor of $\rhobar$.
  \end{thm}
  \begin{proof}
    For each $v\in \mathbb{S}$, let the eigenvalues of $\rhobar(\Frob_v)$ be
    $\gamma_{v,1}\ne\gamma_{v,2}$. Let $\fN$ denote the Artin
    conductor of $\rhobar$, and $N>3$ an integer prime to $p$ and divisible by $\fN$. By Theorem \ref{thm: lifting to weight p
      by blgg}, we see that for each subset $I\subset S$, there is a
    mod $p$ Hilbert modular eigenform $f_I$ of weight $p$ and level
    $\Gamma_1(N)$ such that $\rhobar_{f_I}\cong\rhobar$, and for each prime
    $v\in \mathbb{S}$, we have $T_vf=\gamma_{v,1}f$ if $v\in I$, and
    $T_vf=\gamma_{v,2}f$, otherwise. Since $\rhobar_{f_I}\cong\rhobar$,
    we see that for each prime $\gerl \nmid Np$ of $F$, the $f_I$ are
    eigenvectors for $T_\gerl$ with eigenvalues $\lambda_\gerl$ which are
    independent of $I$. By a standard argument, using Proposition 2.3
    of \cite{MR507462}, we can furthermore assume (at the possible
    cost of passing to forms of level $N^2$) that for each prime
    $\gerl|N$, we have $T_\gerl f_I=0$ for all $I$. 

    We can and do assume that each $f_I$ is normalised, in the sense
    that (in the notation of \cite{MR507462}) $c(\cO_F,f_I)=1$. For
    any $I\subset \mathbb{S}$, let $\gamma_I=\Pi_{v \in I}
    \gamma_{v,1} \Pi_{v \not\in I} \gamma_{v,2}$; this is   the $T_p$-eigenvalue of $f_I$. Set
    \[
    f=\sum_{I\subset S}(-1)^{|I|}\gamma_I f_I,
    \]
    \[
    g=\sum_{I\subset S}(-1)^{|I|}f_I.
    \]

We begin with a Lemma.

\begin{lemma}\label{Lemma: q-expansion} The section $\pi_1^\ast f-{\rm pr}^\ast \pi_2^\ast g$ of $\underline{\omega}^p$ on $\widetilde{Y}$ has $q$-expansion divisible by $p$ at every cusp of the form $Ta_{\gera,c}$.
\end{lemma}

\begin{proof} We let $\eta$ denote a generator of the sheaf $\underline{\omega}$ on the base of $Ta_\gera$ or $Ta_{\gera,c}$. We first remark that by \cite[(2.23)]{MR507462}, if $h$ is a normalized Hilbert modular eigenform of
  parallel weight $k$, and $h(Ta_\gera)=\sum_{\xi \in (\gera^{-1})^+}
  c_\xi q^\xi\eta^k$, then $c_\xi=c(\xi\gera,h)$ is the eigenvalue of the $T_{\xi\gera}$ operator on $h$, for all $\xi\in (\gera^{-1})^+$.

 Write $f(Ta_{\gera})=\sum_{\xi \in (\gera^{-1})^+} a_\xi(\gera) q^\xi \eta^p$ and $g(Ta_{\gera})=\sum_{\xi \in (\gera^{-1})^+} b_\xi(\gera) q^\xi \eta^p$. We have:
\[
\pi_1^\ast f(Ta_{\gera,c})=f(Ta_{\gera})=\sum_{\xi \in (\gera^{-1})^+} a_\xi(\gera) q^\xi \eta^p,
\]
\[
{\rm  pr}^\ast \pi_2^\ast g(Ta_{\gera,c})={\rm pr}^\ast g(Ta_{p\gera})=\sum_{\xi \in p^{-1}(\gera^{-1})^+} b_\xi(p\gera) q^\xi \eta^p.
\]
It is therefore enough to show that $b_\xi(p\gera)=0$ if $\xi \in
p^{-1}(\gera^{-1})^+-(\gera^{-1})^+$, and that $a_\xi(\gera)\equiv
b_\xi(p\gera)\ {\rm mod}\ p$ for $\xi \in (\gera^{-1})^+$. 

For the first statement, let $v\in \mathbb{S}$ be such that $v$ does not divide $\xi p\gera$. Then we can write 
\[
b_\xi(p\gera)=\sum_{I\subset \mathbb{S}} (-1)^{|I|} c(\xi p\gera,f_I)=\sum_{v \not\in I}(-1)^{|I|}c(\xi p\gera,f_I) -\sum_{v \in I} (-1)^{|I|}c(\xi p\gera,f_I)=0,
\]
using $c(\xi\p\gera,f_I)=c(\xi p\gera,f_{I \cup \{v\}})$, since $v$
does not divide $\xi p \gera$. 

For the second statement, note firstly that  for $\xi \in (\gera^{-1})^+$ we can write
\[
b_\xi(p\gera)=\sum_{I \subset \mathbb{S}} (-1)^{|I|}c(\xi
p\gera,f_I),\]\[a_\xi(\gera)= \sum_{I \subset \mathbb{S}}
(-1)^{|I|}\gamma_I c(\xi \gera,f_I).
\]

Now, if $h$ is a Hilbert modular eigenform, then for any integral
ideal $\germ$ of $\calO_F$, we have $c(p\germ,h)\equiv c((p),h)
c(\germ,h)\ {\rm mod}\ p$. Since for any $I \subset \mathbb{S}$,  we
have $c((p),f_I)=\gamma_I$, the result follows.
\end{proof}
For any section $h \in H^0(\widetilde{X},\underline{\omega}^k)$, we denote its image in $H^0(\widetilde{\Xbar}, \underline{\omega}^k)$ by $\bar{h}$.
\begin{cor}\label{Corollary: equality of section on Xbar} We have the following equality of sections of $\underline{\omega}^p$ on $\widetilde{\Xbar}$:
\[
s^\ast \pi_2^\ast \bar{f}=s^\ast w^\ast {\rm pr}^\ast \pi_2^\ast \bar{g}.
\]
\end{cor}

\begin{proof} Reducing the equation in Lemma \ref{Lemma: q-expansion}
  mod $p$, we obtain an equality of sections on every irreducible
  component of $\widetilde{\Ybar}$ which contains the reduction of a
  cusp of the form $Ta_{\gera,c}$.  These are exactly the irreducible
  components of $ws(\widetilde{\Xbar})$, since
  $w(Ta_{\gera,c})=(Ta_{p\gera},<\!\zeta\!>)=s(Ta_{p\gera})$ (where
  $<\!\zeta\!>$ is the $\calO_F$-module generated by a $p$-th root of
  unity $\zeta$). Pulling back under $w\circ s$, we obtain the desired equality on $\widetilde{\Xbar}$.
\end{proof}

For any scheme $Z$ over $\kappa$, let ${\rm Fr}: Z  \arr Z^{(p)}$ denote the relative Frobenius morphism.  Since $\widetilde{\Xbar}$ has a model over $\FF_p$, we have $\widetilde{\Xbar}^{(p)}=\widetilde{\Xbar}$. For any non-negative integer $k$, we define a morphism 
\[
V: H^0(\widetilde{\Xbar},\underline{\omega}^k) \rightarrow H^0(\widetilde{\Xbar},\underline{\omega}^{kp}) 
\]
as follows: choose a trivialization $\{(U_i,\eta_i)\}$ for $\underline{\omega}$ on $\widetilde{\Xbar}$. Let $f\in H^0(\widetilde{\Xbar},\underline{\omega}^k)$ be given by $f_i \eta_i^k$ on $U_i$. Then, there is a unique section $V(f)$ in $H^0(\widetilde{\Xbar},\underline{\omega}^{kp})$ whose restriction to $U_i$ is ${\rm Fr}^\ast (f_i) \eta_i^{kp}$. 

Calculating on points, we see easily that for $\widetilde{\Xbar}$, we
have $\pi_2\circ s={\rm Fr}$. Let $\bf{h}$ denote the Hasse invariant
of parallel weight $p-1$. It can be defined as follows: let $U$ be an open subset of $\Xbar$ over which $\omega$ is trivial, and let $A_U$ denote the universal abelian scheme over $U$. Let  $\eta$ be a non vanishing section of $\omega$ on $U$; it can be thought of as a section of $\Omega_{A_U/U}$. We let $\eta^{(p)}$ denote the induced section of $\Omega_{A_U^{(p)}/U}$. Let ${\rm Ver}: A_U^{(p)} \rightarrow A_U$ be the Verschiebung morphism. Then, there is a unique $\lambda \in \calO_{\Xbar}(U)$, such that  ${\rm Ver}^\ast \eta=\lambda \eta^{(p)}$. We define a section of $\omega^{p-1}$ on $U$ via 
\[
{\bf h}_{U,\eta}:=\lambda\eta^{p-1}.
\]
It is easy to see that there is a unique section of $\omega^{p-1}$ on $\Xbar$, denoted ${\bf h}$, such that ${\bf h}_{|_U}={\bf h}_{U,\eta}$ for any choice of $(U,\eta)$ as above. See \cite[\S 7.11]{AG} for an equivalent construction.

\begin{prop} We have $V(\bar{f})=V({\bf h})\bar{g}$ as sections of $\underline{\omega}^{p^2}$ on $\Xbar$. Furthermore, $\bar{f}$ is divisible by ${\bf h}$, and $\bar{f}/{\bf h}$ is a mod $p$ Hilbert modular form of parallel weight one defined over $\kappa$.
\end{prop}

\begin{proof} Let $U$ be an open subset of $\Xbar$ over which $\underline{\omega}$ is trivializable, and $\eta$ a non-vanishing section of $\underline{\omega}$ over $U$. We claim that if ${\bf h}=\lambda \eta^{p-1}$ on $U$, then
\[
\lambda s^\ast \pi_2^\ast \eta= s^\ast w^\ast {\rm pr}^\ast \pi_2^\ast \eta.
\]
Evaluating both sides at a point corresponding to $\uA$, we need to show that for the natural projection 
\begin{eqnarray}\label{Equation: Verschiebung projection} 
{\rm pr}: A/{{\rm Ker}({\rm Frob}_A)} \rightarrow A/A[p] \cong A,
\end{eqnarray}
we have ${\rm pr}^\ast \eta=\lambda \eta^{(p)}$, which follows from the definition of the Hasse invariant, since ${\rm pr}$ is the Verschiebung morphism of $A$. 

Now, writing $\bar{f}=F \eta^p$ and $\bar{g}=G\eta^p$ on $U$, and using the above claim, over $U$, we can write 
\[
\lambda^p s^\ast \pi_2^\ast \bar{f}=(s^\ast \pi_2^\ast F)(\lambda^p s^\ast \pi_2^\ast \eta^p)=\Fr^\ast(F)( s^\ast w^\ast {\rm pr}^\ast \pi_2^\ast \eta^p).
\]
On the other hand, we have
\[
\lambda^ps^\ast w^\ast {\rm pr}^\ast \pi_2^\ast \bar{g}=\lambda^p(s^\ast w^\ast \pi_2^\ast G) (s^\ast w^\ast {\rm pr}^\ast \pi_2^\ast \eta^{p})=\lambda^pG(s^\ast w^\ast {\rm pr}^\ast \pi_2^\ast \eta^{p}).
\]
At a point corresponding to an {\it ordinary} abelian variety $A$, the section $s^\ast w^\ast {\rm pr}^\ast \pi_2^\ast \eta$ specializes to  $\pr^\ast \eta$, where $\pr$ is as in (\ref{Equation: Verschiebung projection}), and is, hence, non vanishing. Corollary \ref{Corollary: equality of section on Xbar} now implies  that over $\Xbar^{ord} \cap U$ we have
\[
{\rm Fr}^\ast(F)=\lambda^p G={\rm
  Fr^\ast}(\lambda)G,\]
the last equality because $\bf h$ is defined on a model of  $\Xbar$ over $\FF_p$.  Running over a trivializing open covering of $\Xbar$ for $\omega$, we conclude that $V(\bar{f})=V({\bf
  h})\bar{g}$ on $\Xbar^{ord}$. Since $\Xbar^{ord}$ is Zariski dense in $\Xbar$, it follows that  
\[
 V(\bar{f})=V({\bf h})\bar{g}
 \] 
 as sections of $\omega^{p^2}$ over $\Xbar$. 

Viewing $F/\lambda$ as a function on the ordinary part of $U$, we need
to show that it extends to all of $U$. Since $U$ is smooth over
$\kappa$,  it is enough to show that the Weil divisor of $F/\lambda$
is effective. But the coefficients appearing in that divisor are  the
coefficients of the Weil divisor of $G$ multiplied by $p$. Since $G$
has an effective Weil divisor, so does $F/\lambda$, and hence $F$ is
divisible by $\lambda$ on $U$. Repeating this argument over an open
covering of $\widetilde{\Xbar}$, we obtain that $\bar{f}$ is divisible
by $\bf h$,  and $\bar{f}/{\bf h}$ is a mod $p$ Hilbert modular form
of parallel weight one defined over $\kappa$, as required.
\end{proof}

We can now finish the proof of  Theorem \ref{thm: main result}. The desired mod $p$ Hilbert modular form of  parallel weight one $h$ is  $\bar{f}/{\bf h}$. Since ${\bf h}$ has $q$-expansion $1$ at all unramified cusps, it follows that $h$ satisfies the desired assumptions.
\end{proof}

\section{Serre's Conjecture implies Artin's Conjecture}\label{sec:
  Artin}\subsection{}In this final section we generalise the arguments of
\cite{MR1434905} to show that for a fixed totally real field $F$, the
weak form of Serre's conjecture for $F$ implies the strong form of
Artin's conjecture for two-dimensional totally odd representations
over $F$. To
be precise, the weak version of Serre's conjecture that we have in
mind is the following (cf. Conjecture 1.1 of \cite{bib:BDJ}, where it
is described as a folklore conjecture).
\begin{conj}\label{conj:Serre}
  Suppose that $\rhobar:G_F\to\GL_2(\Fpbar)$ is continuous,
  irreducible and totally odd. Then $\rhobar\cong\rhobar_f$ for some
  Hilbert modular eigenform $f$.
\end{conj}
Meanwhile, we have the following strong form of Artin's conjecture.
\begin{conj}
  \label{conj:Artin}Suppose that $\rho:G_F\to\GL_2(\C)$ is continuous,
  irreducible and totally odd. Then $\rho\cong\rho_f$ for the some
  Hilbert modular eigenform $f$ (necessarily of parallel weight one).
\end{conj}
In order to show that Conjecture \ref{conj:Serre} implies Conjecture
\ref{conj:Artin}, we follow the proof of Proposition 1 of
\cite{MR1434905}, using Theorem \ref{thm: main result} in place of the
results of Gross and Coleman--Voloch used in \cite{MR1434905}. The
argument is slightly more involved than in \cite{MR1434905}, because
we have to be careful to show that the $p$-distinguishedness
hypothesis in Theorem \ref{thm: main result} is satisfied.
\begin{thm}\label{thm: Serre implies Artin}
  Fix a totally real field $F$. Then Conjecture \ref{conj:Serre}
  implies Conjecture \ref{conj:Artin}.
\end{thm}
\begin{proof}Suppose that $\rho:G_F\to\GL_2(\C)$ is continuous,
  irreducible and totally odd. Then $\rho(G_F)$ is finite, so after
  conjugation we may assume that $\rho:G_F\to\GL_2(\cO_K)$, where
  $\cO_K$ is the ring of integers in some number field $K$. We will
  show that there are a fixed integer $N$ and infinitely many rational
  primes $p$ such that for each such $p$, if $\rhobar_p$ denotes the
  reduction of $\rho$ mod $p$ (or rather, modulo a prime of $\cO_K$
  above $p$), then $\rhobar_p$ arises from the reduction mod $p$ of the
  Galois representation associated to an eigenform in
  $S_1(\Gamma_1(N),\C)$, the space of cuspidal Hilbert modular forms of parallel weight one and level $\Gamma_1(N)$ over $\CC$. Since
  $S_1(\Gamma_1(N),\C)$ is finite-dimensional, there are only finitely many
  such eigenforms, so we see that one eigenform $f$ must work for
  infinitely many $p$; but then it is easy to see that
  $\rho\cong\rho_f$, as required.

  Firstly, we claim that it suffices to prove that there is a fixed
  $N$ and infinitely many primes $p$ such that $\rhobar\cong\rhobar_f$
  for some eigenform $f\in S_1(\Gamma_1(N),\Fpbar)$ (using the
  notation of Section \ref{section: weight one}). To see this, note
  that for all but finitely many primes $p$, the finitely generated $\Z$-module
  $H^1(X,\underline{\omega})$ is $p$-torsion free, so that for all but
  finitely many $p$ the reduction map
  $H^0(X,\underline{\omega})\to
  H^0(\Xbar,\underline{\omega})$ is surjective, and the
  Deligne--Serre lemma (Lemma 6.11 of \cite{deligne-serre}) allows us
  to lift from $S_1(\Gamma_1(N),\Fpbar)$ to $S_1(\Gamma_1(N),\C)$.

We are thus reduced to showing that there are infinitely many primes
$p$ for which $\rhobar_p$ satisfies the hypotheses of Theorem
\ref{thm: main result}. Firstly, note that there is at most one prime
$p$ for which $\rho|_{G_{F(\zeta_p)}}$ is reducible, so if we exclude
any such prime, as well as the (finitely many) primes dividing
$\#\rho(G_F)$, the primes which ramify in $F$, and the
primes less than $7$, then $\rhobar_p$ will satisfy the requirements
of the first paragraph of Theorem \ref{thm: main result}.
  
If we also exclude the finite many primes $p$ for which
$\rho|_{G_{F_v}}$ is ramified for some $v|p$, we see that it is enough
to show that there are infinitely many $p$ such that for all $v|p$,
$\rhobar_p(\Frob_v)$ has distinct eigenvalues.

In fact, we claim that it is enough to see that there are infinitely
many $p$ such that for all $v|p$, $\rho(\Frob_v)$ is not scalar. To
see this, suppose that $\rho(\Frob_v)$ is not scalar, but
$\rhobar_p(\Frob_v)$ is scalar. Then it must be the case that the
difference of the eigenvalues of $\rho(\Frob_v)$ is divisible by some
prime above $p$. Now, there are only finitely many non-scalar elements
in $\rho(G_F)$, and for each of these elements there are only finitely
many primes dividing the difference of their eigenvalues, so excluding
this finite set of primes gives the claim.

Let $\proj \rho$ be the projective representation
$\proj\rho:G_F\to\PGL_2(\C)$ obtained from $\rho$. We must show that
there are infinitely many primes $p$ such that for each place $v|p$ of
$F$, $\proj\rho(\Frob_v)\ne 1$. Letting
$L=\overline{F}^{\ker\proj\rho}$, we must show that there are
infinitely many primes $p$ such that no place $v|p$ of $F$ splits
completely in $L$. Let $M$ be the normal closure of $F$ over $\Q$, and
$N$ the normal closure of $L$ over $\Q$. Since $\rho$ is totally odd,
we see that $M$ is totally real and $N$ is totally imaginary. Consider
a complex conjugation $c\in\Gal(N/\Q)$. By the Cebotarev density
theorem there are infinitely many primes $p$ such that $\Frob_p$ is
conjugate to $c$ in $\Gal(N/\Q)$, and it is easy to see that each such
prime splits completely in $M$ and thus in $F$, and that no place $v|p$ of $F$ splits
completely in $L$, as required.
\end{proof}
\bibliographystyle{amsalpha} 

\bibliography{companionformsweightone} 

\

\end{document}